\documentclass[10pt]{amsart}

\usepackage{url,amssymb,enumerate,colonequals}

\usepackage{mathrsfs} 
\usepackage[section]{placeins}
\usepackage{MnSymbol}
\usepackage{extarrows}
\usepackage{lscape}
\usepackage[all,cmtip]{xy}
\usepackage{tikz-cd}

\usepackage[OT2,T1]{fontenc}
\DeclareSymbolFont{cyrletters}{OT2}{wncyr}{m}{n}
\DeclareMathSymbol{\Sha}{\mathalpha}{cyrletters}{"58}

\usepackage[OT2,T1]{fontenc}
\usepackage{color}

\usepackage[
        colorlinks, citecolor=darkgreen,
        backref,
        pdfauthor={Katerina Santicola},
]{hyperref}
\usepackage[backrefs,lite]{amsrefs} 

\usepackage{comment}
\usepackage{multirow}

\newcommand{\PP}{\mathbb{P}}
\newcommand{\Q}{\mathbb{Q}}

\newcommand{\Z}{\mathbb{Z}}

\newcommand{\cB}{\mathscr{B}}

\newcommand{\GL}{\operatorname{GL}}

\newtheorem{thm}{Theorem}

\newtheorem{lem}[thm]{Lemma}
\newtheorem{conj}[thm]{Conjecture}

\newtheorem{prop}[thm]{Proposition}

\theoremstyle{definition}

\newtheorem{defn}[equation]{Definition}

\theoremstyle{remark}
\newtheorem{remark}[thm]{Remark}

\definecolor{darkgreen}{rgb}{0,0.5,0}

\DeclareRobustCommand{\SkipTocEntry}[5]{}

\begin{document}

\title[]{
	Curves with prescribed rational points
}

\begin{abstract}
Given a smooth curve $C/\mathbb{Q}$ with genus $\geq 2$, we know by Faltings'
Theorem that $C(\Q)$ is finite. 
Here we ask the reverse
question: given a finite set of rational points $S\subseteq \PP^n(\Q)$, does
there exist a smooth curve $C/\Q$ contained in $\PP^n$
such that $C(\Q)=S$? We answer this question in the
affirmative by providing an effective algorithm for constructing such a curve.
\end{abstract}

\author{Katerina Santicola}

\address{Mathematics Institute\\
    University of Warwick\\
    CV4 7AL \\
    United Kingdom}

\email{Katerina.Santicola@warwick.ac.uk}

\date{\today}
\thanks{The author's research is funded by
a doctoral studentship from the Heilbronn Institute
for Mathematical Research}

\maketitle

\section{Introduction}

The genus $g$ of a smooth projective curve $C/\Q$ determines the possibilities for
its set of rational points. If $g=0$, then $C(\Q)=\emptyset$ or $C(\Q)$ 
is infinite.
If $g=1$, then $C(\Q)=\emptyset$, or by the Mordell--Weil theorem,
the set $C(\Q)$ can be given the structure of a finitely generated abelian group.
Mazur's torsion theorem \cite{mazur}
restricts the possible torsion subgroups of elliptic curves. As a consequence,
for example, there are no genus $1$ curves $C/\Q$ with exactly $11$ rational
points.
For $g \ge 2$, Faltings' theorem \cite{faltings} asserts that $C(\Q)$ is finite. The Uniformity conjecture of Caporaso, Harris, and Mazur \cite{uniformity}, predicts that $\# C(\Q)$ is
uniformly bounded as $C$ ranges through curves $C/\Q$ of fixed genus $g \ge 2$.

It is therefore natural to ask about the possible restrictions on $\# C(\Q)$
as $C$ ranges through all curves $C/\Q$ of all possible genera $g \ge 2$. Poonen \cite{poonen} has shown that every value of $\#C(\Q)$ is possible\footnote{In fact, that  every value of $\#C(k)$ is possible for any global field $k$.}. In this paper, we show that there are no restrictions on the set $C(\Q)$
itself, beyond finiteness.
\begin{thm}\label{thm:main}
Let $n \ge 2$. 
    Given any finite set of rational points $S\subset \PP^n(\Q)$, there exists a smooth projective curve $C/\Q$ contained
in $\PP^n$, of genus $g \ge 2$, such that $C(\Q)=S$. 
\end{thm}

Our proof gives an effective algorithm for computing the curve $C$ from the given set of rational points $S$. It builds on earlier constructions due to Gajovi\'c \cite{gajovic}
and the current author \cite{me}. Let $\mathcal{P}_R$ denote the perfect powers of the ring $R$:
\[
\mathcal{P}_R=\{ a^m: a\in R,m\geq 2\}.
\]
They show that for a finite set of perfect powers $S\subset \mathcal{P}_R$, for $R=\Z$ and $R=\Q$ respectively, there is a polynomial $f_S(X)\in \Z[X]$ such that $f_S(X)\cap \mathcal{P}_R = S$.  We will show how the construction of the polynomial $f_S$ in those papers can be adapted to give a construction of a suitable curve $C$ in our context.

\section{Superelliptic Curves}

The proof of Theorem \ref{thm:main} involves constructing suitable 'superelliptic' curves, and then gluing them together. By a \textit{superelliptic curve $C$ over $\Q$} \cite{arithmetic} we mean a smooth projective curve in $\PP^2$ associated to the affine model given by
$$C:\; y^m=f(x),$$
where $f$ is separable of degree $d\geq 3$ and $m\geq 2$ is an integer. If $m=2$ and $d=3$ then $C$ is an elliptic curve, 
whereas if $m=2$ and $d \geq 5$ then $C$ is hyperelliptic. 
When $m=d$, the genus $g$ of the curve is $g=(d-1)(d-2)/2$; thus if $d \ge 4$,
then $g \ge 2$ and so $C(\Q)$ is finite by Faltings' theorem \cite{faltings}. 

We begin by assuming our prescribed set of points $S$ satisfies some nice conditions.
\begin{defn}
Let $S=\{(a_1,b_1),\dots ,(a_r,b_r)\} \subset \mathbb{A}^2(\Q)$ be a finite set of distinct rational points in the affine plane.
We say that $S$ is \textit{acceptable} if the following conditions are satisfied: 

\begin{enumerate}[(i)]
    \item if $a_i=a_j$ then $b_i=b_j$,
    \item $a_i,b_j \neq 0$ for all $i,j$,
    \item $a_i,b_j\in \Z$ for all $i,j$.
\end{enumerate}
\end{defn}
For any such $S$, we can construct a superelliptic curve $C$ whose set of affine rational points is exactly $S$,
and such that $C$ has no rational points at infinity. In Section \ref{section:dimensionn}, we will show how to extend the argument to sets in $\PP^n(\Q)$.

\begin{thm}\label{thm:theone}
Given an acceptable set $S\subset \mathbb{A}^2(\Q)$, and a finite subset $\mathcal{B} \subset \mathbb{C}$,
 there exists a separable polynomial $f_{S,\mathcal{B}}(X)\in \Q[X]$ of some degree $d \ge 4$, not vanishing at any point in $\mathcal{B}$,
such that $C(\Q)=S$, where
$C$ is the superelliptic curve associated to the affine model
\[
    C \; : \; y^d=f_{S,\mathcal{B}}(X).
\]
\end{thm}

Our general argument will not apply to the case when the set of prescribed points $S$ has size $|S|=1$. In this case, we can take
any curve $C/\Q$ of genus $g \ge 2$ with one rational point. A
suitable change of coordinates will send the point in $S$ to the point on $C$. For instance, take the hyperelliptic curve over $\Q$ with affine model
\[
C:y^2=x^5-2.
\]
We verify in Magma \cite{magma} that the rational points on the Jacobian, $J(\Q)$, form a group of rank 0 with trivial torsion subgroup. Since $C(\Q)$ injects into $J(\Q)$, we conclude that $C$ has no rational points other than the one at infinity.

From now on we assume that $S=\{(a_1,b_1),\dots ,(a_r,b_r)\}$ is acceptable and $|S|=r>1$. Set 
\[
d=18r+3,
\]
and let
\[
m:=\prod_{j<k}^r (a_j-a_k).
\]

Let $L(X)\in \Q[X]$ be the Lagrange interpolation polynomial passing through the points $S'$, where
\[
S'=\{(a_i,b_i^d)\mid (a_i,b_i)\in S\},
\]
which is of the form
\[ 
    L(X)=\sum_{j=1}^r b_j^d\prod_{\substack{0\leq k \leq r\\ k\neq j}}\frac{X-a_k}{a_j-a_k}.
\]

\begin{lem}\label{lem:sep}
Let $S=\{(a_1,b_1),\dots,(a_r,b_r)\}$ be an acceptable set of points of size
$|S|=r$. Let $n$ be any integer such that $n\geq r$. 
Then there exists an irreducible polynomial
\[
h(X)=\sum_{i=1}^n \alpha_i X^i \in \Q[X]
\]
satisfying the following conditions
\begin{enumerate}[(i)]
\item $h(a_i)=b_i^{d}$ for all $(a_i,b_i)\in S$,
\item $h(X)$ has degree $n$,
\item $m \cdot h(X) \in \Z[X]$.
\end{enumerate}
\end{lem}

\begin{proof}[Proof of Lemma \ref{lem:sep}]
Let
\[
    b(X)=\prod_{\substack{i=1}}^r(X-a_i)\in \Z[X].
\]
Since $S$ is acceptable, $b_j\neq 0$ for all $j$, so $L(X)$ and $b(X)$ are coprime. Choose $c(X)\in \Z[X]$ to be monic and coprime to $L(X)$ of degree $n-r$, and let
\[
    R(X,Y)=L(X)+b(X)c(X)Y \in \Q(Y)[X].
\]
This is irreducible in $\Q(Y)[X]$.
By Hilbert's
irreducibility theorem \cite{serre}, there is a thin set $A\subset
\PP^1(\Q)$ such that if $t\notin A$, then $R(X,t)$ is irreducible over
$\Q$.  It follows that there exist infinitely many integers $y \notin A$.
Choose any such $y$ and let $h(X)=R(X,y)$. Then $h(X)$ will satisfy the conditions above. 
\end{proof}

Although the polynomial $h(X)$ interpolates the points of $S$, we cannot be sure that its associated superelliptic curve does not have other rational points. Our approach will be to build, from $h(X)$, a superelliptic curve $C/\Q$ of degree $d$ such that all rational points on $C$ map to the rational points of an elliptic curve $E/\Q$. This elliptic curve will be chosen to have rank 0, and this will allow us to conclude there are no other rational points on $C$. 

\section{Twists of the Mordell curve}

Let $E$ be an elliptic curve over $\Q$. Let $rk$ denote its Mordell-Weil rank, so that $E(\Q)\cong E(\Q)_{tors}\times \Z^{rk}$. Let $E_k$ denote the quadratic twist of $E$ by  $\Q(\sqrt{k})$. A celebrated conjecture of Goldfeld's predicts that as $k$ varies, the rank should be 0 half of the time, and 1 the other half. Namely,
\begin{conj}[Goldfeld]\cite{goldfeld}
Let $E_k$ be the quadratic twist of $E$ by $k$. Then
\[
\lim_{X \rightarrow \infty}\frac{\sum_{|k|<X}rk(E_k)}{\#\{k:|k|<X\}}=\frac{1}{2}.
\]
\end{conj}
In our construction, we will relate the rational points on our superelliptic curve to the rational points on a quadratic twist $E_p$ of the Mordell curve $E$ by a prime $p$, where
\[
E:Y^2=X^3-1,
\]
and
\[
E_{p}:pY^2=X^3-1.
\]
We want to choose $p$ so that $rk(E_p)=0$. Juyal, Moody and Roy \cite{phillipe} use the technique of 2-descent to show it is sufficient to choose $p\equiv 5 \mod 12$.

\begin{thm}\cite{phillipe}\label{lem:twist}
Let $E:Y^2=X^3-1$ be the Mordell curve, and let 
\[
E_{p}:pY^2=X^3-1
\]
denote its twist by a prime $p$. Then for all primes $p\equiv 5 \mod 12$, we have
\[
E_{p}(\Q)\cong \{(1,0), \infty\}.
\]
\end{thm}

\section{Choosing a prime}

We continue with our construction of $f_{S,\mathcal{B}}$. Apply Lemma
~\ref{lem:sep} with $n=6r+3$
to produce $h(X)$ of degree $n$. Let
\[
N:=\prod_{v_p(m)>0}p.
\]
With $S=\{(a_1,b_1),\dots, (a_r,b_r)\}$ as before, let
\[
    g(X,t)=tN^6\prod_{i=1}^r(X-a_i)^6+1 \in \Q[X,t].
\]

\begin{prop}\label{prop:mini}
Let $g(X,t)$ as above. Given a finite set of values $\mathcal{B}\subset \mathbb{C}$, the zeroes of $g(X,t_0)$ lie outside of $\mathcal{B}$ for all but finitely many values of $t_0\in \mathbb{C}$.
\end{prop}

\begin{proof}
Every $\alpha\in \mathcal{B}$ is a root of $g(X,t_0)$ for exactly one value $t_0\in \mathbb{C}$.
\end{proof}

\begin{lem}\label{lem:sepf}
Let $g(X,t)$ as above. For all but finitely many primes $p$, 
\[
(h(X)-1)g(X,p)+1
\]
is separable.
\end{lem}
\begin{proof}

First, because $h(X)$ is irreducible, its discriminant is non-zero and has finite support. For those primes $p$ not dividing the discriminant, $h(X)$ is separable modulo $p$.   
Now rewrite the polynomial $(h(X)-1)g(X,t)+1$ as
\begin{equation}\label{eqn:polygh}
(h(X)-1) g(X,t)+1 \; = \; h(X) \; + \; t H(X) \; = \; \sum_{i=0}^{6r+n} c_i X^i.
\end{equation}
where $H(X)$ is some polynomial in $\Q[X]$. For all but finitely many values of $p$, we can specialise $g(X,t)$ at $t=p$ so that the valuations of the coefficients are:
\[
v_p(c_i)=\begin{cases}
0 & 0 \le i \le n,\\
1 & n+1 \le i \le 6r+n.
\end{cases}
\]
Thus the Newton polygon of the polynomial \eqref{eqn:polygh}
consists of two segments of lengths $n$ and $6r$
and slopes $0$ and $1/6r$ respectively.
By the theory of Newton polygons, \cite{Gouvea}*{7.2.3} 
\eqref{eqn:polygh} will factor as the product of two irreducible polynomials in $\Z_{p}[X]$:
\[
    (h(X)-1)g(X,p)+1 \; =\; F_1(X)F_2(X)
\]
 corresponding to the two line segments of the Newton polygon. 
Here $F_1(X)$ has degree $n$ and its roots have valuation $0$,
and $F_2(X)$ has degree $6r$ and its roots have valuation $1/6r$.
In particular, $F_1$, $F_2$ are coprime.
Moreover, $F_1 \equiv h(X) \mod{p}$. However $h(X)$
is separable modulo $p$, and thus $F_1$ is separable in $\Z_p[X]$, for all but finitely many $p$.
The second factor $F_2$ is irreducible because it is pure of slope
$1/{6r}$. This means their product $F_1(X)F_2(X)$ is separable for all but finitely many primes $p$.
\end{proof} 

With all of this in mind, we will choose a prime $\ell$ that satisfies a list of conditions.

\begin{prop}\label{prop:choice}
Given a finite set $\mathcal{B}\subset \mathbb{C}$, we can choose a prime $\ell$ such that
\begin{enumerate}[(i)]
    \item the zeroes of $g(X,\ell)$ lie outside of $\mathcal{B}$,
    \item $(h(X)-1)g(X,\ell)+1$ is separable,
    \item $(\ell,N)=1$,
    \item $\ell \equiv 5\mod 12$.
\end{enumerate} 
\end{prop}

\begin{proof}
    This is always possible by Proposition \ref{prop:mini}, Lemma \ref{lem:sepf} and Dirichlet's theorem. 
\end{proof}

\section{Proof of Theorem \ref{thm:theone}}

Let $\ell$ be a prime chosen as in Proposition \ref{prop:choice}. Let
\[
g(X):=g(X,\ell).
\]
Let
\[
f_{S,\mathcal{B}}(X):=g(X)(h(X)-1)g(X)+1),
\]
and let
\[
    C: y^d=f_{S,\mathcal{B}}(X).
\]
We note that 
\[
    f_{S,\mathcal{B}}(a_i)=h(a_i)=b_i^d
\]
as required. Moreover, 
\[
d=\deg f_{S,\mathcal{B}}(X)=2\deg g(X)+\deg h(X)=3(6r+1).
\]
Note that the leading coefficient of $f_{S,\mathcal{B}}(X)$ is such that
\[
v_\ell(a_d)=2,
\]
therefore there are no rational points at infinity.

\begin{prop}\label{prop:smooth}
The curve $C: y^d=f_{S,\mathcal{B}}(X)
$
is smooth.
\end{prop}

\begin{proof}
We verify that $f_{S,\mathcal{B}}(X)$ is separable. Note that $g(X)$ and $(h(X)-1)g(X)+1$ are coprime in $\Q[X]$, so it is enough to show that each
is separable. We have shown the latter is separable in Lemma \ref{lem:sepf}, and $g(X)$ is irreducible by examining the Newton polygon with respect to $v_{\ell}$.     
\end{proof}

Now that we have shown the curve is smooth, we want to show its set of rational points is exactly $S$. So, we want to show that if $f_{S,\mathcal{B}}(x)=y^d$ for some $x,y\in \Q$, then $(x,y)\in S$.

\begin{prop}\label{prop:val}
If $(x,y)\in C$, then $g(x)=\beta^{3}$ for some $\beta\in \Q^*$.
\end{prop}

\begin{proof}
Suppose 
\[
f_{S,\mathcal{B}}(x)=y^d
\]
for some $x,y\in \Q$. First, let $p$ be a prime such that $p\neq \ell$. Assume $v_p(x)<0$. As the $a_i$ are all integral, we have $\sum v_p(x-a_i)=rv_p(x)$. As $0\leq v_p(\ell N) \leq 1$ and $r>1$, this means $v_p(g(x))<0$. Then 
\begin{align*}
    v_p(g(x))=v_p(\ell)+6v_p( N)+6\sum_{i=0}^r v_p(x-a_i)=6v_p(N)+6rv_p(x)&\equiv 0 \mod 3.
\end{align*}

Now, let $p$ be a prime such that $v_p(g(x))\geq  0$. By the above, we must have that $v_p(x)\geq 0$. Moreover, if $v_p(m)>0$ then $v_p(g(x))=0$ by our choice of $N$.
Now, assume $v_p(g(x))>0$. We have that $v_p(h(x))=v_p(mh(x))$. Since $v_p(x)\geq 0$ and $mh(X)\in \Z[X]$, we have that $v_p(mh(x))\geq 0$. From this we conclude that 
\[
v_p((h(x)-1)g(x)+1)=0.
\]
Therefore $v_p(g(x))\equiv 0\mod d$, and since $3\mid d$ we are done.

Now, consider $v_{\ell}(g(x))$. If $v_{\ell}(x)>0$, then $v_{\ell}(g(x))=0$. If $v_{\ell}(x)<0$, then
\[
v_l(h(x)-1)=nv_{\ell}(x),
\]
because $v_{\ell}(\alpha_n)=0$, where $\alpha_n$ is the leading coefficient of $h(X)$ (this follows from the proof of Lemma \ref{lem:sepf}). As before we have that $v_{\ell}(g(x))=1+6rv_{\ell}(x)$, and so 
\[
v_\ell(g(x)((h(x)-1)g(x)+1))=2+dv_{\ell}(x)\equiv 0 \mod d,
\]
which is not possible. So, $v_{\ell}(g(x))=0$.
\end{proof}

\begin{lem}\label{lem:zero2}
The polynomial $f_{S,\mathcal{B}}(X)$ has no rational roots.
\end{lem} 
\begin{proof}[Proof of Lemma \ref{lem:zero2}]

Note that $g(x)=0$ is not possible since $g(x)>0$, so we must have that $(h(x)-1)g(x)+1=0$, and
\[
    v_p(h(x)-1)+v_p(g(x))=0.
\]
If $v_p(g(x))>0$, then as in the proof of Proposition \ref{prop:val}, we have $v_p(h(x)-1)\geq 0$. So for all $p$ we have $v_p(g)\leq 0$ and $v_p(x)\leq 0$. This means $v_p(h(x)-1)\geq 0$ for all $p$, and therefore $h(x)-1\in \Z$. Then, as $(h(x)-1)g(x)=-1$, we must have $g(x)=1$, so $x=a_i$ for some $i$. Since $S$ is acceptable, and $f_{S,\mathcal{B}}(a_i)=b_i^d$, we reach a contradiction.

\end{proof}

\begin{proof}[Proof of Theorem \ref{thm:theone}]

Recall that we supposed that $f_{S,\mathcal{B}}(x)=y^d$ for some rational $y$.
To complete the proof of Theorem~\ref{thm:theone}, we want to show that $x=a_i$ for some $i$.
Lemma~\ref{lem:zero2} establishes this if $f_{S,\mathcal{B}}(x)=0$.
Thus we may suppose $f_{S,\mathcal{B}}(x) \ne 0$. From Proposition \ref{prop:val}, we see that $g(x)=\beta^3$ for some $\beta\in \Q^*$. So, we have a point on the twist of the Mordell curve
\begin{equation}\label{eqn:thetwo}
    E_{\ell}:\ell Y^2=X^3-1,
\end{equation}
with $X\neq 0$, and $\ell \equiv 5 \mod 12$. By Theorem \ref{lem:twist}, $E_{\ell}(\Q)=\{(1,0),\infty\}$, and we conclude that $(X,Y)=(1,0)$. Therefore $g(x)=1$ and $x=a_i$ for some $i$. Then 
\[
f_{S,\mathcal{B}}(x)=f_{S,\mathcal{B}}(a_i)=b_i^d,
\]
and since $d$ is odd, we conclude that $y=b_i$. 

\end{proof}

\section{Proof of Theorem \ref{thm:main}}\label{section:dimensionn}

We have shown that given an acceptable set $S\subset \mathbb{A}^2(\Q)$, and a finite subset $\mathcal{B} \subset \mathbb{C}$,
 there exists a separable polynomial $f_{S,\mathcal{B}}(X)\in \Q[X]$ of some degree $d \ge 4$, not vanishing at any point in $\mathcal{B}$,
such that its associated superelliptic curve has rational point set $C(\Q)=S$.
In this section, we show how we can generalise this to sets in $\PP^2(\Q)$, and glue together the superelliptic curves to generalise the result to finite sets in $\PP^n(\Q)$.

\begin{remark}\label{remark:poonen}
If, in the statement of Theorem \ref{thm:main}, we had not asked that $C$ be a projective curve, then one could simply use Bertini's theorem \cite{hartshorne} to find a smooth curve $C$ of large genus passing through all points of $S$. Then, by taking hyperplane sections, remove from $C$ the finitely many points outside of $S$.
\end{remark}

\begin{lem}\label{lem:canchange}
    Let $S=\{P_1,\dots ,P_r\}\subseteq \mathbb{A}^n(\Q)$ be a finite set of points, 
which are distinct and non-zero. Then there exists a change of coordinate matrix $A\in \GL_n(\Q)$ such that
    \begin{enumerate}[(a)]
        \item for all $i$, $j$, we have $x_i(AP_j)\neq 0$,
        \item for all $i$, $j$, $k$, $\ell$ with $i \ne k$ or $j \ne \ell$, we have $x_i(AP_j)\neq x_k(AP_{\ell})$.
    \end{enumerate}
where the $x_i$ are the usual coordinate functions.
\end{lem}
\begin{proof}
We write $M_{n \times n}$ for the affine space of dimension $n^2$ considered as the space of $n \times n$ matrices.
Consider the following subvarieties of $M_{n \times n}$:
\begin{itemize}
\item $U=\{ A \in M_{n \times n} : \det(A) =0\}$,
\item $V_{i,j}=\{ A \in M_{n \times n} : x_i(AP_j)=0\}$ with $1 \le i \le n$ and $1 \le j \le r$,
\item $W_{i,j,k,\ell}=\{ A \in M_{n \times n} : x_i(A P_j) = x_k(A P_{\ell})\}$ with $1 \le i,  k \le n$ and $1 \le j , \ell \le r$
and either $i \ne k$ or $j \ne \ell$.
\end{itemize}
It is easy to see that these are all proper subsets of $M_{n \times n}$ which are closed in the Zariski topology.
As the rational points on $M_{n \times n}$ are Zariski dense, it follows that there is some $A \in M_{n \times n}(\Q)$
that does not lie on any of the $U$, $V_{i,j}$, $W_{i,j,k,\ell}$. 

\end{proof}

Using Lemma \ref{lem:canchange}, we are able to deduce Theorem \ref{thm:main} from Theorem \ref{thm:theone}.

\begin{proof}[Proof of Theorem \ref{thm:main}]
Let $S$ be the finite set of rational points in $\PP^n(\Q)$. By finiteness
of this set, we can always find a hyperplane on which none of the points lie.
After a suitable change of coordinates, all the points lie in an affine chart.
We may assume from Lemma~\ref{lem:canchange} that the set $S=\{P_1,\dotsc,P_r\}$ satisfies the conclusion of Lemma \ref{lem:canchange}, scaling the points if necessary to ensure integrality.

For $2 \le j \le n$, let
\[
S_j=\{(x_1(P_1),x_j(P_1)) \, , \, (x_1(P_2),x_j(P_2)) \, , \,  \dotsc \, , \, (x_1(P_r),x_j(P_r))\} \; \subset \; \mathbb{A}^2(\Q).
\]
Observe that these are acceptable sets.
We construct sequences of finite subsets of $\mathbb{C}$:
\[
\cB_2 \subset \cB_3 \subset \cdots \subset \cB_n
\]
and a sequence of separable polynomials $f_2,f_3,\dotsc,f_n \in \Q[x_1]$ as follows. We start by letting $\cB_2=\emptyset$.
We let $f_2=f_{S_2,\cB_2}$ constructed as in Theorem~\ref{thm:theone}. For the inductive step of the construction,
let $\cB_j$ be the union of $\cB_{j-1}$ and the set of complex roots of $f_{j-1}$, and we let $f_j=f_{S_j,\cB_j}$.
Let $C \subset \PP^n$ be the projective closure of the affine variety 
\[
C \; = \; V\left( \, x_2^{d} - f_2(x_1) \, , \, x_3^{d} - f_3(x_1) \, , \, \cdots \, , \, x_n^{d}-f_n(x_1) \, \right) \; \subset \mathbb{A}^n.
\]
Note that $C$ has function field 
\[
\Q(C) = \Q\left( \sqrt[d]{f_2(X)},\sqrt[d]{f_3(X)}, \dotsc,\sqrt[d]{f_n(X)} \right)
\]
which is a finite extension of $\Q(X)$. Thus $C$ is an absolutely irreducible curve.
It follows from Theorem~\ref{thm:theone} that $C(\Q)=\{P_1,P_2,\dots,P_r\}$. Moreover, we note
that the construction ensures that $f_1,\dotsc,f_n$ have pairwise distinct complex roots. It follows from this,
using the Jacobian criterion,
that the curve $C$ is smooth.

\end{proof}

\section*{Acknowledgements} The author would like to thank Samir Siksek for his continuous support, as well as Ross Paterson and Michael Stoll for helpful conversations. The author would also like to thank Bjorn Poonen for pointing out Remark \ref{remark:poonen}.

\subsection*{Data availability} Data sharing is not applicable to this article as no datasets were generated or
analysed in the current paper.

\bibliographystyle{plain} 
\bibliography{References} 

\end{document}